\documentclass[11pt]{article}
\usepackage[a4paper,margin=28mm]{geometry}
\usepackage[T1]{fontenc}
\usepackage{lmodern}
\usepackage{microtype}
\usepackage{amsmath,amssymb,amsthm,mathtools}
\usepackage{booktabs}
\usepackage{hyperref}
\usepackage{mathrsfs}
\newtheorem{theorem}{Theorem}[section]
\newtheorem{proposition}[theorem]{Proposition}
\newtheorem{lemma}[theorem]{Lemma}
\newtheorem{corollary}[theorem]{Corollary}
\theoremstyle{remark}
\newtheorem{remark}[theorem]{Remark}
\newtheorem{conjecture}[theorem]{Conjecture}
\newtheorem{problem}[theorem]{Problem}

\newcommand{\mdr}{\operatorname{mdr}}
\newcommand{\AR}{\operatorname{AR}}
\newcommand{\ii}{\mathrm i}
\newcommand{\Sing}{\operatorname{Sing}}

\title{On the Numerical Terao Conjecture}
\author{Piotr Pokora}
\date{\today}

\begin{document}
\maketitle

\begin{abstract}
We prove that the Numerical Terao Conjecture holds for even-degree conic-line arrangements having only ADE singularities. We then show that the conjecture fails in the broader quasi-homogeneous setting once ordinary quadruple points are admitted, by constructing a degree-nine counterexample consisting of a pair of conic-line arrangements, each with seven lines and one smooth conic, with the same weak combinatorics
\[
W(\mathcal{CL}) = (7,1;\,8A_1+D_4+4X_9).
\]
We show that one curve is free with exponents $(4,4)$, whereas the other is nearly free with exponents $(3,6)$. This yields a counterexample to the strongest known formulation of the Numerical Terao Conjecture \cite{CP}.
\end{abstract}

\section{Introduction}

Let $\mathcal{C} = \{f=0\}$ be a reduced plane curve. We write
\[
 \AR(f)=\{(a,b,c)\in \mathbb C[x,y,z]^{\oplus 3}:\;af_x+bf_y+cf_z=0\}
\]
and denote by $\mdr(f)$ the least degree of a non-zero homogeneous element of $\AR(f)$.  We say that the curve is \textbf{free} with exponents ${\rm exp}(\mathcal{C}) = (d_1,d_2)$ when
\[
 \AR(f)\simeq S(-d_1)\oplus S(-d_2),\qquad d_1+d_2=d-1,
\]
where $d=\deg f$ and $S=\mathbb C[x,y,z]$. Note that ${\rm mdr}(f) = d_{1}$, and we can order exponents $d_{1} \leq d_{2}$.

The weak-combinatorics used below records the number of smooth irreducible
components of each degree and the numbers of singular points of each prescribed
type. In particular, it determines the total Tjurina number
\[
 \tau(\mathcal C)=\sum_{p\in\Sing(\mathcal C)}\tau(\mathcal C,p).
\]
Here we will consider curves admitting ADE singularities and ordinary quadruple points $X_{9}$. Notice that the modulus of an $X_9$-singularity plays no role here: every
ordinary quadruple point has Tjurina number $9$, independently of its modulus.
Thus, in the present setting, the weak combinatorics records only the number
of ordinary quadruple points and not their individual analytic moduli.

The Numerical Terao Conjecture can be formulated as follows.
\begin{conjecture}[Numerical Terao Conjecture]
Let $\mathcal{C}_{1}, \mathcal{C}_{2}$ be two reduced curves in $\mathbb{P}^{2}_{\mathbb{C}}$ with the property that they possess only smooth irreducible components and they admit only quasi-homogeneous singularities. Suppose that $\mathcal{C}_{1}$ and $\mathcal{C}_{2}$ have the same weak-combinatorics and let $\mathcal{C}_{1}$ be free, then $\mathcal{C}_{2}$ is also free.
\end{conjecture}
Let us recall that a singularity is called quasi-homogeneous if and only if there exists a holomorphic change of variables so that the defining equation becomes weighted homogeneous. 

It is known that the Numerical Terao Conjecture is false in the class of line arrangements in the complex projective plane \cite{Mar}. The present note asks what changes when non-linear components are allowed. Its first result gives a positive answer in even degree under an ADE hypothesis. As a corollary we show that maximizing conic-line arrangements form a bounded class of curves, i.e., their total degree is bounded by $18$. The second main result gives a degree-nine counterexample after ordinary quadruple points are allowed. Recall that an ordinary quadruple point, even if it has a modulus, is quasi-homogeneous, and its Tjurina number is always $9$. To the best of our knowledge, this is the first counterexample to the Numerical Terao Conjecture involving a non-linear irreducible component. The present work should be viewed as the culmination of the author’s study of arrangements of lines and one conic with quasi-homogeneous ordinary singularities, initiated in \cite{PokoraOneConic}. All symbolic computations here were performed using \texttt{SINGULAR} \cite{Singular}.

\section{Very concise algebraic preliminaries}

We recall two standard estimates.  Suppose that $ \mathcal{C} = \{f=0\}$ is a reduced plane curve of degree $d$ with only quasi-homogeneous singularities, and let $\alpha_{\mathcal{C}}$ be the smallest Arnold exponent among its singular points. The celebrated Dimca-Sernesi theorem~\cite{DimcaSernesi} gives us the following bound
\begin{equation}\label{eq:DS}
 \mdr(f)\ge \alpha_{\mathcal{C}} \cdot d-2.
\end{equation}
We also use the following du Plessis-Wall upper bound \cite{DPP1} and its freeness characterization in the form recalled in~\cite{DimcaMaxTau,DimcaSticlaru}.
For a reduced plane curve $\mathcal{C} : f=0$ we put $r=\mdr(f)$.  If $r<d/2$, then
\begin{equation}\label{eq:DPW-small}
 \tau(\mathcal{C})\le (d-1)^2-r(d-r-1),
\end{equation}
and equality holds precisely when $\mathcal{C}$ is free.  If $r\ge d/2$, then the
refined bound gives us
\begin{equation}\label{eq:DPW-large}
 \tau(\mathcal{C})\le (d-1)^2-r(d-r-1) -\binom{2r-d+2}{2}.
\end{equation}

\section{A structure theorem on conics and lines with ADE singularities}

\begin{lemma}\label{lem:types}
Let $\mathcal{CL}$ be a reduced plane arrangement of lines and smooth conics having only ADE singularities.  Then every singular point is of one of the following types:
\[
 A_1,A_3,A_5,A_7,\qquad D_4,D_6,D_8,D_{10}.
\]
Consequently,
\[
 \alpha_{\mathcal{CL}}\ge \alpha(D_{10})=\frac59.
\]
\end{lemma}

\begin{proof}
Every local branch of $\mathcal{CL}$ is smooth.  Hence an $A$-type singularity
which can occur in the arrangement has two smooth branches and is of type
$A_{2q-1}$.  The number $q$ is their local intersection multiplicity.  Since
the components have degrees at most two, B\'ezout's theorem gives $q\le4$.
Thus only $A_1,A_3,A_5,A_7$ can occur.

An ADE singularity with three smooth branches is necessarily of type
$D_{2q+2}$: two branches have contact order $q$, and a third branch is
transverse to them.  Again $q\le4$, so the possible types are
$D_4,D_6,D_8,D_{10}$.  The remaining simple singularities have either a
singular local branch or a different branch structure and therefore cannot occur in a union of smooth components. Let us recall from \cite{DimcaSernesi} that for our quasi-homogeneous singularities one has
\[
 \alpha(A_k)=\frac12+\frac1{k+1},
 \qquad
 \alpha(D_k)=\frac{k}{2(k-1)}.
\]
Among the listed types the smallest value is
\[
 \alpha(D_{10})=\frac{10}{18}=\frac59.
\]
\end{proof}

\begin{theorem}\label{thm:even}
Let $\mathcal{CL} = \{f=0\} \subset \mathbb{P}^{2}_{\mathbb{C}}$ be a conic-line arrangement of even degree
$d=2m\ge4$ with only ADE singularities.  Suppose that $\mathcal{CL}$ is free. Then every conic-line arrangement with the same weak combinatorics is free. Consequently, there is no counterexample to the Numerical Terao Conjecture in even degree in the ADE conic-line class.
\end{theorem}

\begin{proof}
Set $r=\mdr(f)$.  Since $\mathcal{CL}$ is free of degree $2m$, its exponents
are $(r,2m-1-r)$, and therefore
\begin{equation}\label{eq:r-upper}
 r\le m-1.
\end{equation}
By Lemma~\ref{lem:types} and \eqref{eq:DS},
\begin{equation}\label{eq:r-lower}
 r\ge \left\lceil\frac{10m}{9}-2\right\rceil.
\end{equation}
For $2\le m\le9$ one has
\[
 \left\lceil\frac{10m}{9}-2\right\rceil=m-1.
\]
Combining this with \eqref{eq:r-upper} shows that
\begin{equation}\label{eq:middle-exp}
 r=m-1,
 \qquad
 \exp(\mathcal C)=(m-1,m).
\end{equation}
If $m\ge10$, then the right-hand side of \eqref{eq:r-lower} is at least
$m$, contradicting \eqref{eq:r-upper}.  Hence no free ADE conic-line
arrangement of even degree $d\ge20$ exists.

It remains to prove combinatorial rigidity in the possible degrees
$4\le d\le18$.  Let $\mathcal C':g=0$ have the same weak combinatorics as
$\mathcal C$, and put $r'=\mdr(g)$.  The two curves have the same total
Tjurina number.  By \eqref{eq:middle-exp},
\begin{equation}\label{eq:tau-even-free}
 \tau(\mathcal C')=\tau(\mathcal C)
 =(2m-1)^2-m(m-1).
\end{equation}
The universal Arnold-exponent estimate again gives $r'\ge m-1$. If $r'=m-1$, then \eqref{eq:tau-even-free} is equality in
\eqref{eq:DPW-small}, and therefore $\mathcal C'$ is free.

Assume now that $r'\ge m$, and write $r'=m+s$ with $s\ge0$.  Applying the
refined bound \eqref{eq:DPW-large}, we obtain
\begin{align*}
 \tau(\mathcal C')
 &\le (2m-1)^2-(m+s)(m-1-s)-\binom{2s+2}{2}\\
 &= (2m-1)^2-m(m-1)-(s+1)^2\\
 &=\tau(\mathcal C)-(s+1)^2,
\end{align*}
which contradicts \eqref{eq:tau-even-free}.  Thus the second case is
impossible, and $\mathcal C'$ is necessarily free.
\end{proof}
The above result gives us a surprising boundedness result on maximizing conic-line arrangements. By \cite[Theorem 2.9]{max}, an ADE reduced plane curve of even degree $d=2m$ is maximizing if and only if it is free with exponents
\[
(m-1,m).
\]
\begin{corollary}
Let $\mathcal{CL}\subset \mathbb{P}^2_{\mathbb{C}}$ be a conic-line arrangement of even degree
\[
d=2m\geq 4
\]
having only ADE singularities. If $\mathcal{CL}$ is maximizing, then
\[
d\leq 18.
\]
\end{corollary}

\begin{proof}
Every maximizing conic-line arrangement is a free
ADE conic-line arrangement. By Theorem \ref{thm:even}, free conic-line arrangements with only ADE singularities are bounded in degree, and necessarily
\[
d\leq 18.
\]
\end{proof}
Finishing this section, we would like to focus on the odd degree case.
\begin{proposition}
\label{odd}
Let $\mathcal{CL} = \{f=0\} \subset \mathbb{P}^{2}_{\mathbb{C}}$ be a conic-line arrangement of odd degree
$$d = 2m+1 \geq 5$$
having only ${\rm ADE}$ singularities. If $\mathcal{CL}$ is free, then
$$d \leq 27.$$
\end{proposition}
\begin{proof}
Let $r=\operatorname{mdr}(f)$. Since $\mathcal{CL}$ is free of degree
$d=2m+1$, its exponents are $(r,2m-r)$, and hence $r\leq m$.
By Lemma~\ref{lem:types} and the Dimca--Sernesi bound, we have
\[
m\geq r\geq \alpha_{\mathcal{CL}}\cdot d-2
\geq \frac{5}{9}(2m+1)-2
=\frac{10m}{9}-\frac{13}{9}.
\]
Consequently $m\leq 13$, and therefore $d=2m+1\leq 27$.
\end{proof}
Let us recall that by \cite[Proposition 5.1(a)]{max}, an ADE reduced plane curve of odd degree $d=2m+1 \geq 5$ is maximizing if and only if it is free with exponents
\[
(m-1,m+1).
\]
\begin{corollary}
Let $\mathcal{CL}\subset \mathbb{P}^2_{\mathbb{C}}$ be a conic-line arrangement of odd degree
\[
d=2m+1\geq 5
\]
having only ADE singularities. If $\mathcal{CL}$ is maximizing, then
\[
d \in \{5,7,9\}.
\]
\end{corollary}
\begin{proof}
Let $r = {\rm mdr}(f)$. Since $\alpha_{\mathcal{CL}} \geq \frac{5}{9}$, we have
$$ m-1 = r \geq \alpha_{\mathcal{CL}}\cdot d -2  \geq  \frac{10m}{9} - \frac{13}{9},$$
hence $m \leq 4$.
\end{proof}
Finally, combining Theorem \ref{thm:even} and Proposition \ref{odd}, we arrive at the following observation.
\begin{corollary}
There are only finitely many weak-combinatorial types of free conic-line arrangements with ${\rm ADE}$ singularities.
\end{corollary}
\begin{proof}
We have already observed that the degree of a free conic-line
arrangement with only ADE singularities is bounded by $27$. Moreover,
by Lemma~\ref{lem:types}, the only possible singularity types are
\[
A_1,\ A_3,\ A_5,\ A_7,\ D_4,\ D_6,\ D_8,\ D_{10}.
\]
For a fixed total degree $d\leq 27$, there are only finitely many
possibilities for the numbers of line and conic components. Furthermore,
the number of singular points of each of the above types is bounded,
for instance by the total intersection number of the irreducible
components. Consequently, only finitely many weak-combinatorial types
can occur.
\end{proof}
\section{On a special pair of conic-line arrangements}

Put $S=\mathbb C[x,y,z]$ and let $\ii^2=-1$.  Define
\begin{align*}
 F={}&y(y-x)(y+x)(y-x-2z)(y-x+2z)\\
 &\qquad\times (y+x-2z)(y+x+2z)
       (x^2-y^2+4\ii yz-4z^2),
\end{align*}
and
\begin{align*}
 G={}&xyz(x-y)(2x-z)(y-z)(2x+y-z)\\
 &\qquad\times \bigl(z^2-yz-2xz+(1-\ii)xy\bigr).
\end{align*}
We set $\mathcal C_F : F=0$ and $\mathcal C_G : G=0$. Obviously the two quadratic factors define smooth conics. 
\begin{theorem}\label{thm:counterexample}
The curves $\mathcal C_F$ and $\mathcal C_G$ have the same weak
combinatorics, namely
\[
 W(\mathcal C_F)=W(\mathcal C_G)
 =(7,1;\,8A_1+D_4+4X_9).
\]
Moreover, $\mathcal C_F$ is free with exponents $(4,4)$, whereas
$ \mdr(G)=3$ and $\mathcal{C}_G$ is nearly free with exponents $(3,6)$.
Thus the pair is a counterexample to the Numerical Terao Conjecture in the class of conic-line arrangements with ordinary quasi-homogeneous singularities. 
\end{theorem}
\begin{proof}
The seven-line part of $F$ has six ordinary double points and the following five ordinary triple points:
\[
 (0:0:1),\quad
 (1:0:-\tfrac12),\quad
 (1:0:\tfrac12),\quad
 (1:1:0),\quad
 (1:-1:0).
\]
Its conic passes transversely through the last four points and it avoids the first one and all six double points. Thus the first point remains a
$D_4$-point, whereas the other four triple points become ordinary quadruple
points.  The only unused conic-line intersection multiplicities occur on
$y-x=0$ and $y+x=0$; their residual points are
\[
 (1:1:\ii),\qquad (1:-1:-\ii),
\]
and both are transverse.

The seven-line part of $G$ likewise has six ordinary double points and five ordinary triple points. Four of the triple points are
\[
 (0:1:0),\quad (0:1:1),\quad (1:0:0),\quad (1:0:2),
\]
and the conic passes transversely through all the four. The remaining triple point is $(0:0:1)$, which the conic avoids. The only line with unused conic-intersection degree is $x-y=0$.  Restriction of the conic to this line gives
\[
 z^2-3yz+(1-\ii)y^2,
\]
whose discriminant is $5+4\ii\ne0$.  Hence it contributes two distinct
transverse nodes. It follows in both cases that
\[
 n_2=8,\qquad n_3=1,\qquad n_4=4,
\]
or, in singularity notation according to Arnold's catalogue \cite{arnold},
\[
 8A_1+D_4+4X_9.
\]
The total Tjurina number is therefore
\begin{equation}\label{eq:tau-counterexample}
 \tau=8\cdot1+1\cdot4+4\cdot9=48.
\end{equation}
All singularities of the two curves are quasi-homogeneous and the smallest Arnold exponent is now the value $1/2$ which corresponds to $X_9$. Hence
\begin{equation}\label{eq:lower-X9}
 \mdr(H)\ge 9\cdot\frac12-2=\frac52,
\text{ and it implies that }\mdr(H)\ge3
\end{equation}
for $H \in \{F,G\}$.

Now we observe that there is an explicit cubic Jacobian relation for $G$, and this computation can be performed using \texttt{SINGULAR}. Set
\begin{align*}
 a={}&-x(4x-5y)(2x+y-2z),\\
 b={}& y(5x-4y)(2x+y-2z),\\
 c={}& z(10x^2-12xy-xz+5y^2-yz).
\end{align*}
A direct check shows that
\[
 aG_x+bG_y+cG_z=0.
\]
Together with \eqref{eq:lower-X9}, this proves
\[
 \mdr(G)=3.
\]
For $d=9$ and $r=3$, the maximal Tjurina number is
\[
 \tau_{\max}(9,3)=8^2-3(8-3)=49,
\]
and such a curve attaining the maximal values of $\tau_{\max}(9,3)$ would be an $\mathscr{M}$-arrangement of conics and lines in the sense of \cite{MJ}. By \eqref{eq:tau-counterexample} we have
\[
 \tau(\mathcal C_G)=48=\tau_{\max}(9,3)-1.
\]
Using \cite{DimcaMaxTau} we can conclude that $\mathcal C_G$ is nearly free with exponents $(3,6)$, and in particular is not free.

Let us look at $\mathcal{C}_{F}$. Using \texttt{SINGULAR}, we can directly check that $\AR(F)_3=0$, and this means that
\[
 \mdr(F)\geq 4.
\]
We set 
\[
\begin{aligned}
a &= x^{4}-10x^{2}y^{2}+9y^{4}
     +28x^{2}z^{2}-36y^{2}z^{2},\\
b &= -8x^{3}y+8xy^{3}-8xyz^{2},\\
c &= 10x^{3}z-10xy^{2}z-8xz^{3}.
\end{aligned}
\]
These polynomials satisfy the Jacobian syzygy

\[
aF_x+bF_y+cF_z=0,
\]
hence ${\rm mdr}(F)=4$.
Finally,
\[
 \tau(\mathcal C_F)=48=8^2-4(8-4),
\]
so equality holds in \eqref{eq:DPW-small}.  It follows that $\mathcal C_F$
is free with exponents $(4,4)$.
\end{proof}
\begin{remark}
The existence of a free conic-line arrangement with weak combinatorics
\[
W(\mathcal{C})=(7,1;8A_1+D_4+4X_9)
\]
was already established in \cite[Theorem~1.3]{PokoraOneConic}. For the reader's convenience, we present here a substantially simpler defining equation.
\end{remark}
Concluding this note we would like to ask the following question.
\begin{problem}
Does there exist a counterexample to the Numerical Terao Conjecture over $\mathbb{R}$, that is, in $\mathbb{P}^{2}_{\mathbb{R}}$?
\end{problem}
\section*{Funding}
Piotr Pokora is supported by the National Science Centre (Poland) Sonata Bis Grant  \textbf{2023/50/E/ST1/00025.} For the purpose of Open Access, the author has applied a CC-BY public copyright license to any Author Accepted Manuscript (AAM) version arising from this submission.

Piotr Pokora\\
\noindent
Department of Mathematics,
University of the National Education Commission Krakow,
Podchor\c a\.zych 2,
PL-30-084 Krak\'ow, Poland. \\
Email: \url{piotr.pokora@uken.krakow.pl}.
\end{document}